\documentclass[12pt,reqno]{amsart}

\usepackage{color} %Packages Used%
\usepackage{amstext} \usepackage{amsthm} \usepackage{amsmath} \usepackage{amssymb}
\usepackage{latexsym} \usepackage{amsfonts} \usepackage{graphicx} \usepackage{texdraw} \usepackage{graphpap}
\usepackage{enumerate}

\usepackage[pagebackref,hypertexnames=false, colorlinks, citecolor=red, linkcolor=red]{hyperref}
\usepackage[backrefs]{amsrefs}

\usepackage[inline,nomargin]{fixme}

\input txdtools

\begin{document} %Commands Used%
\newcommand{\ci}[1]{_{ {}_{\scriptstyle #1}}}

\newcommand{\norm}[1]{\ensuremath{\left\|#1\right\|}} \newcommand{\abs}[1]{\ensuremath{\left\vert#1\right\vert}}
\newcommand{\ip}[2]{\ensuremath{\left\langle#1,#2\right\rangle}} \newcommand{\p}{\ensuremath{\partial}}
\newcommand{\pr}{\mathcal{P}}

\newcommand{\pbar}{\ensuremath{\bar{\partial}}} \newcommand{\db}{\overline\partial} \newcommand{\D}{\mathbb{D}}
\newcommand{\B}{\mathbb{B}}   \newcommand{\Sn}{{\mathbb{S}_n}} \newcommand{\T}{\mathbb{T}} \newcommand{\R}{\mathbb{R}}
\newcommand{\Z}{\mathbb{Z}} \newcommand{\C}{\mathbb{C}} \newcommand{\N}{\mathbb{N}} 

\newcommand{\td}{\widetilde\Delta}

\newcommand{\Aa}{\mathcal{A}} \newcommand{\BB}{\mathcal{B}} \newcommand{\HH}{\mathcal{H}} \newcommand{\KK}{\mathcal{K}} \newcommand{\DD}{\mathcal{D}}
\newcommand{\LL}{\mathcal{L}} \newcommand{\MM}{\mathcal{M}}  \newcommand{\FF}{\mathcal{F}}  \newcommand{\GG}{\mathcal{G}}  \newcommand{\TT}{\mathcal{T}}
 \newcommand{\UU}{\mathcal{U}}

\newcommand{\AapBerg}{\Aa_p(\B_n)} \newcommand{\AaFock}{\Aa_\phi (\C^n)} \newcommand{\AatwoBerg}{\Aa_2(\B_n)}
\newcommand{\Om}{\Omega} \newcommand{\La}{\Lambda} \newcommand{\AaordFock}{\Aa (\C^n)}

\newcommand{\rk}{\operatorname{rk}} \newcommand{\card}{\operatorname{card}} \newcommand{\ran}{\operatorname{Ran}}
\newcommand{\osc}{\operatorname{OSC}} \newcommand{\im}{\operatorname{Im}} \newcommand{\re}{\operatorname{Re}}
\newcommand{\tr}{\operatorname{tr}} \newcommand{\vf}{\varphi} \newcommand{\f}[2]{\ensuremath{\frac{#1}{#2}}}

\newcommand{\kzp}{k_z^{(p,\alpha)}} \newcommand{\klp}{k_{\lambda_i}^{(p,\alpha)}} \newcommand{\TTp}{\mathcal{T}_p}

\newcommand{\vp}{\varphi} \newcommand{\al}{\alpha} \newcommand{\be}{\beta} \newcommand{\la}{\lambda}
\newcommand{\tf}{\tilde{f}}
\newcommand{\li}{\lambda_i} \newcommand{\lb}{\lambda_{\beta}} \newcommand{\Bo}{\mathcal{B}(\Omega)}
\newcommand{\Bbp}{\mathcal{B}_{\beta}^{p}} \newcommand{\Bbt}{\mathcal{B}(\Omega)} \newcommand{\Lbt}{L_{\beta}^{2}}
\newcommand{\Kz}{K_z} \newcommand{\kz}{k_z} \newcommand{\Kl}{K_{\lambda_i}} \newcommand{\kl}{k_{\lambda_i}}
\newcommand{\Kw}{K_w} \newcommand{\kw}{k_w} \newcommand{\Kbz}{K_z} \newcommand{\Kbl}{K_{\lambda_i}}
\newcommand{\kbz}{k_z} \newcommand{\kbl}{k_{\lambda_i}} \newcommand{\Kbw}{K_w} \newcommand{\kbw}{k_w}
\newcommand{\BL}{\mathcal{L}\left(\mathcal{B}(\Omega), L^2(\Om;d\sigma)\right)}
\newcommand{\Fpphi}{\ensuremath{{\mathcal{F}}_\phi ^p }}
\newcommand{\Ftwophi}{\ensuremath{{\mathcal{F}}_\phi ^2 }}
\newcommand{\incn}{\ensuremath{\int_{\C}}}
\newcommand{\Finfphi}{\ensuremath{\mathcal{F}_\phi ^\infty }}
\newcommand{\Fp}{\ensuremath{\mathcal{F} ^p }} \newcommand{\Fq}{\ensuremath{\mathcal{F} ^q }}
\newcommand{\Ft}{\ensuremath{\mathcal{F} ^2 }} \newcommand{\Lt}{\ensuremath{L ^2 }}
\newcommand{\Lp}{\ensuremath{L ^p }}
\newcommand{\Fonephi}{\ensuremath{\mathcal{F}_\phi ^1 }}
\newcommand{\Lpphi}{\ensuremath{L_\phi ^p}}
\newcommand{\Ltwophi}{\ensuremath{L_\phi ^2}}
\newcommand{\Lonephi}{\ensuremath{L_\phi ^1}}
\newcommand{\af}{\mathfrak{a}} \newcommand{\bb}{\mathfrak{b}} \newcommand{\cc}{\mathfrak{c}}
\newcommand{\Fqphi}{\ensuremath{\mathcal{F}_\phi ^q }}

\newcommand{\supp}{\operatorname{supp}}
\newcommand{\spn}{\operatorname{span}}
\newcommand{\llan}{\left\langle}
\newcommand{\rran}{\right\rangle}
\newcommand{\llv}{\left\lvert}
\newcommand{\rrv}{\right\rvert}
\newcommand{\llV}{\left\lVert}
\newcommand{\rrV}{\right\rVert}
\newcommand{\tld}{\widetilde}

%%%%%%%%%%%%%%%%%%%%%%%%%%%%

\newcommand{\entrylabel}[1]{\mbox{#1}\hfill}

\newenvironment{entry} {\begin{list}{X}%
  {\renewcommand{\makelabel}{\entrylabel}%
      \setlength{\labelwidth}{55pt}%
      \setlength{\leftmargin}{\labelwidth}%\labelsep}%
      \addtolength{\leftmargin}{\labelsep}%
   }%
}% {\end{list}}

%%%%%%%%%%%%%%%%%%%%%%%%%%%%

\numberwithin{equation}{section}

\newtheorem{thm}{Theorem}[section] \newtheorem{lem}[thm]{Lemma} \newtheorem{cor}[thm]{Corollary}
\newtheorem{conj}[thm]{Conjecture} \newtheorem{prob}[thm]{Problem} \newtheorem{prop}[thm]{Proposition}
\newtheorem*{prop*}{Proposition}

\theoremstyle{remark} \newtheorem{rem}[thm]{Remark} \newtheorem*{rem*}{Remark} \newtheorem{example}[thm]{Example}

\theoremstyle{definition} \newtheorem{definition}[thm]{Definition}

\title{Density results for continuous frames}

\author[M. Mitkovski]{Mishko Mitkovski$^\dagger$} \address{Mishko Mitkovski, Department of Mathematical Sciences\\
Clemson University\\ O-110 Martin Hall, Box 340975\\ Clemson, SC USA 29634} \email{mmitkov@clemson.edu} \thanks{$\dagger$ Research supported in part by National Science Foundation
DMS grant \# 1600874.}

\author[A. Ramirez]{Aaron Ramirez} \address{Aaron Ramirez, Department of Mathematical Sciences\\
Clemson University\\ O-110 Martin Hall, Box 340975\\ Clemson, SC USA 29634}


\begin{abstract}  
We derive necessary conditions for localization of continuous frames in terms of generalized Beurling densities. As an important application we provide necessary density conditions for sampling and interpolation in a very large class of reproducing kernel Hilbert spaces. 

\end{abstract}

\maketitle

\section{Introduction}

A well-known elementary linear algebra fact says that any linear independent set of vectors in a finite-dimensional vector space cannot have more elements than any spanning set. In particular, the cardinality of any Riesz sequence cannot be greater than the cardinality of any frame. Even though there is no exact analog of these results in the infinite dimensional setting there are many well-known results which are very similar in spirit. In this infinite-dimensional setting one needs to replace the comparison of cardinalities with a more suitable concept - which is the concept of densities. Basically one needs to compare the cardinalities locally everywhere and then take the appropriate limits. First density results were obtained in the context of non-harmonic complex exponentials. The first definitive results were proved by Beurling~\cite{Beu} and Kahane~\cite{Kah} who characterized frames and Riesz sequences of complex exponentials in terms of certain natural densities of their frequency 
sequence. These densities are now known as Beurling densities. These results were later extended and generalized in various ways and to many different settings~\cites{AB, Bar, Lan, HNP, OS, RS, Seip, Seip1, Seip2, SW}. The most important and popular approaches for proving the necessary part of density theorems are due to Landau~\cite{Lan}, Ramanthan and Steger~\cite{RS}, Balan et al~\cite{BCHL1, BCHL2}, and the recent one of Nitzan and Olevski~\cite{NO}. 

Most density results for sampling and interpolation pertain to a specific reproducing kernel Hilbert space (RKHS). Very recently two universal results appeared~\cites{Fuhr, Aaron} providing a necessary density theorem for a very general class of reproducing kernel Hilbert spaces. Both of these results use similar (but non-equivalent) set of assumption to deduce essentially same conclusions. In this paper we provide a universal density theorem that implies both of these results. Our result can be viewed as a starting point for many density theorems.

\section{Preliminaries}

Let $\HH$ be a Hilbert space. A collection of vectors $\left\{ f_x  \right\}_{x\in \left( X,d,\mu\right)} \subseteq \mathcal{H}$ indexed by a metric measure space $(X, d, \mu)$ (with metric $d$ and a Borel measure $\mu$) is called a \emph{continuous frame} if there exist $0<c\leq C<\infty$ such that 
\[ c\norm{f}^2\leq \int_X\abs{\ip{f}{f_x}}^2d\mu(x)\leq C\norm{f}^2,\] for all $f\in \HH$. 

A continuous frame is said to be a \emph{continuous Parseval frame} in the case when $c=C=1$. 
The name continuous frames is used to stress the analogy with the usual (discrete) frames. Namely, for $X=\N$ (with the usual metric and the counting measure) this definition reduces to the usual definition of frames. Even though the concept of continuous frames has been around for quite some time now, see, e.g.,~\cite{AAG, FR}, so far there is no established standard terminology and other names for the same notion can be found in the literature, e.g., ``continuous resolution of the identity'', ``generalized coherent states'', etc. 

As in the discrete case one can define a frame operator $S:\HH\to \HH$ by $Sf=\int_X\ip{f}{f_x}f_xd\mu(x)$. Here and elsewhere the integral of a Hilbert space-valued function will be defined in the weak sense. For example, $Sf$ is the unique element in $\HH$ such that $$\ip{Sf}{g}=\int_X \ip{f}{f_x}\ip{f_x}{g}d\mu(x),$$ for all $g\in\HH$. The existence and uniqueness of this element is guaranteed by the Riesz representation theorem. The canonical dual continuous frame is defined by $\tf_x=S^{-1}f_x, x\in X$. It is easy to see that Parseval continuous frames  coincide with their duals, i.e., $f_x=\tilde{f}_x$.

For the reader's convenience, in the following lemma, we have collected few simple preliminary facts that will be used throughout the paper. The proofs are straightforward, so we omit them.

\begin{lem}\label{easy} Let $\FF, \GG\subseteq \HH$ be closed subspaces of a Hilbert space $\HH$. %Denote by $P_\KK:\HH\to \KK$ the orthogonal projection onto $\KK$. 
Let $\left\{ f_x  \right\}_{x\in \left( X,d,\mu\right)} \subseteq \mathcal{F}$ be a continuous frame for $\mathcal{F}$, and $\left\{ g_x  \right\}_{x\in \left( X,d,\nu\right)} \subseteq \mathcal{G} $ be a continuous frame for $\mathcal{G}$.

\begin{itemize}
\item[(i)] The following formula holds for the orthogonal projection $P_\FF:\HH\to\FF$ onto $\FF$, $$\displaystyle P_{\mathcal{F}}f = \int_X \llan f,\tld{f_x} \rran f_x d\mu(x) = \int_X \llan f,f_x \rran \tld{f_x} d\mu(x),$$ for any $f\in\HH$. 

\item[(ii)] If $a\leq \llan P_{\mathcal{G}} \tld{f_y},f_y \rran \leq b$ for all $y \in \supp{\mu}$, then $$\displaystyle a\mu\left(\Omega\right) \leq \int_X\int_\Omega \left\langle g_x,f_y \right\rangle \left\langle \widetilde{f_y}, \widetilde{g_x} \right\rangle d\mu(y) d\nu(x)\leq b\mu(\Omega),$$ for any Borel set $\Om\subseteq X$.\\ The same inequalities hold under the assumption $a\leq \llan P_{\mathcal{G}} f_y, \tld{f_y}\rran \leq b$ for all $y \in \supp{\mu}$

\end{itemize}

\end{lem}

\subsection{Density of measures}
We define the analog of the Beurling densities replacing the counting measure with a general Borel measure. 

\begin{definition}
Let $\mu$ and $\nu$ be two Borel measures on the same metric space $(X,d)$. Assume that there exists large enough $R>0$ such that $\nu(B(a,R))>0$ for all $a\in X$. We define the upper $D^+_\nu (\mu)$ and the lower density  $D^-_\nu (\mu)$ of $\mu$ with respect to $\nu$ by
\[ D^+_\nu (\mu):=\limsup_{r\rightarrow \infty} \sup_{a\in X} \frac{\mu\left( B(a,r) \right)}{\nu\left( B(a,r) \right)}, \hspace{.5cm} D^-_\nu (\mu):=\liminf_{r\rightarrow \infty} \inf_{a\in X} \frac{\mu\left( B(a,r) \right)}{\nu\left( B(a,r) \right)}.\]

\end{definition}
The classical Beurling densities are recovered by taking the metric space $(X, d)$ to be $\R$ equipped with the usual Euclidean metric, the measure $\nu$ to be the Lebesgue measure $m$, and $\mu$ to be the counting measure of the sequence $\Lambda\subset \R$ whose density we are computing:

\[D^+(\Lambda):=\limsup_{r\rightarrow \infty} \sup_{a\in X} \frac{\#\left(\Lambda\cap B(a,r) \right)}{m\left( B(a,r) \right)}, \hspace{.5cm} D^-(\Lambda):=\liminf_{r\rightarrow \infty} \inf_{a\in X} \frac{\#\left(\Lambda\cap B(a,r) \right)}{m\left( B(a,r) \right)}.\]

\section{Density results for continuous frames}

The following comparison identity represents the basis of our approach. Let $\mathcal{F}, \mathcal{G} \subseteq \mathcal{H}$ be two closed subspaces of a given Hilbert space $\mathcal{H}$. Let $\left\{ f_x  \right\}_{x\in \left( X,d,\mu\right)}$ be a continuous frame for $\mathcal{F}$, and $\left\{ g_x  \right\}_{x\in \left( X,d,\nu\right)}$ be a continuous frame for $\mathcal{G}$. 

\begin{lem}\label{mainlemma} For any Borel set $\Omega\subseteq X$ the following equality holds
\begin{eqnarray*}
\int_\Om \llan P_{\mathcal{G}}\tld{f_y}, f_y \rran d\mu(y) &=& \int_\Om  \llan P_{\mathcal{F}} g_x, \tld{g_x} \rran  d\nu(x) -\int_{\Om^c} \int_\Om \llan g_x, f_y \rran \llan \tld{f_y}, \tld{g_x} \rran d\nu(x) d\mu(y) \\
&&+ \int_{\Om^c} \int_{\Om} \llan g_x, f_y \rran \llan \tld{f_y}, \tld{g_x} \rran d\mu(y) d\nu(x).
\end{eqnarray*}
\end{lem}

\begin{proof} Using that $\{g_x\}_{x\in \left( X,d,\nu\right)}$ is a continuous frame for $\GG$, by Lemma~\ref{easy}, we obtain  

\begin{eqnarray*}
\int_\Om \llan P_{\mathcal{G}}\tld{f_y}, f_y \rran d\mu(y) &=& \int_\Om \int_X \llan g_x, f_y \rran \llan \tld{f_y}, \tld{g_x} \rran d\nu(x) d\mu(y)\\
&=& \int_X \int_\Om \llan g_x, f_y \rran \llan \tld{f_y}, \tld{g_x} \rran d\nu(x) d\mu(y) -\int_{\Om^c} \int_\Om \llan g_x, f_y \rran \llan \tld{f_y}, \tld{g_x} \rran d\nu(x) d\mu(y)\\
&&+ \int_\Om \int_{\Om^c} \llan g_x, f_y \rran \llan \tld{f_y}, \tld{g_x} \rran d\nu(x) d\mu(y)\\
\end{eqnarray*}
Using that $\{f_x\}_{x\in \left( X,d,\mu\right)}$ is a continuous frame for $\FF$ we can apply Lemma~\ref{easy} again (and Fubini's theorem) to obtain the desired equality.
\end{proof}

\begin{thm}
\label{main}
Let $\FF$ and $\GG$ be closed subspaces of a Hilbert space $\HH$. Let $\left\{ f_x  \right\}_{x\in \left( X,d,\mu\right)}$ and  $\left\{ g_x  \right\}_{x\in \left( X,d,\nu\right)}$ be continuous frames for $\mathcal{F}$ and $\mathcal{G}$ respectively satisfying the following localization property:
\begin{itemize}
\item[(L)] For any $\epsilon>0$ there exists $R>0$ such that for all $r\geq R$ and $B=B(a,r)$ we have

$$\abs{ \int_{B^c} \int_{B} \llan g_x, f_y \rran \llan \tld{f_y}, \tld{g_x} \rran d\mu(y) d\nu(x)-\int_{B^c} \int_{B} \llan g_x, f_y \rran \llan \tld{f_y}, \tld{g_x} \rran d\nu(x) d\mu(y)}<\epsilon(\mu+\nu)(B).$$  

\end{itemize}
Then the following hold:
\begin{itemize}
\item[(i)] If $ \llv \llan P_{\mathcal{F}} g_x, \tld{g_x} \rran \rrv \leq 1$ for all $x\in \supp{\left(\nu\right)}$, then
\begin{eqnarray*}
\displaystyle D^-_\mu(\nu) & \geq & \liminf_{r\rightarrow \infty} \inf_{a\in X} \frac{1}{\mu\left( B(a,r) \right)} \llv \int_{B(a,r)}  \llan P_{\mathcal{G}} \tld{f_y},f_y \rran  d\mu(y) \rrv , \\
\displaystyle D^+_\mu(\nu) & \geq & \limsup_{r\rightarrow \infty} \sup_{a\in X} \frac{1}{\mu\left( B(a,r) \right)} \llv \int_{B(a,r)}  \llan P_{\mathcal{G}} \tld{f_y},f_y \rran  d\mu(y) \rrv.
\end{eqnarray*} 
 \item[(ii)]  If $ \left\langle P_{\mathcal{G}}\tld{f_y}, f_y \right\rangle \geq 1$ for all $y\in \supp{\left(\mu\right)}$, then
\begin{eqnarray*}
\displaystyle  D^+_\nu(\mu) & \leq & \limsup_{r\rightarrow \infty} \sup_{a\in X} \frac{1}{\nu\left( B(a,r) \right)} \llv \int_{B(a,r)}  \llan P_{\mathcal{F}} g_x, \tld{g_x} \rran d\nu(x) \rrv , \\
\displaystyle D^-_\nu(\mu) & \leq & \liminf_{r\rightarrow \infty} \inf_{a\in X} \frac{1}{\nu\left( B(a,r) \right)} \llv \int_{B(a,r)}  \llan P_{\mathcal{F}} g_x, \tld{g_x} \rran d\nu(x) \rrv.
\end{eqnarray*} 

\item[(iii)] If $ \llv \llan P_{\mathcal{F}} g_x, \tld{g_x} \rran \rrv \leq 1$ for all $x\in \supp{\left(\nu\right)}$, and $ \llan P_{\mathcal{G}}\tld{f_y}, f_y \rran \geq 1$ for all $y\in \supp{\left(\mu\right)}$, then

\[ D^-_\mu(\nu) \geq 1 , \hspace{1cm} D^+_\nu(\mu) \leq 1.\]

\end{itemize}
\end{thm}

\begin{proof}
 
Let $\epsilon>0$. By Lemma~\ref{mainlemma} and property (L) we can find $R>0$ such that for all $r\geq R$ we have 
\[ \llv \int_{B(a,r)} \llan P_{\mathcal{G}}\tld{f_y}, f_y \rran d\mu(y)  -  \int_{B(a,r)}  \llan P_{\mathcal{F}} g_x, \tld{g_x} \rran  d\nu(x) \rrv \leq \epsilon(\mu(B(a,r))+\nu(B(a,r))). \]

Using the assumption in (i) we obtain 
\[\llv \int_B \llan P_{\mathcal{G}}\tld{f_y}, f_y \rran d\mu(y) \rrv \leq  (1+\epsilon)\nu(B(a,r))+\epsilon\mu(B(a,r)).\]

Therefore, \[\inf_{a\in X} \frac{1}{\mu(B(a,r))} \llv \int_B \llan P_{\mathcal{G}}\tld{f_y}, f_y \rran d\mu(y) \rrv \leq (1+\epsilon)\inf_{a\in X}\frac{\nu(B(a,r))}{\mu(B(a,r))}+\epsilon.\] This proves the inequality in (i).

The proof of (ii) is similar, and (iii) is just a combination of (i) and (ii). 
\end{proof}

The applicability of the previous result depends heavily on how easy is to verify the localization condition (L).  The next two results provide conditions which imply (L) and are simpler to verify. 

\begin{prop}\label{nd} Let $\FF$ and $\GG$ be closed subspaces of a Hilbert space $\HH$. Let $\left\{ f_x  \right\}_{x\in \left( X,d,\mu\right)}$ and  $\left\{ g_x  \right\}_{x\in \left( X,d,\nu\right)}$ be continuous frames for $\mathcal{F}$ and $\mathcal{G}$ respectively. If

\begin{itemize}
\item[(i)] Both continuous frames are bounded, i.e., $\displaystyle \sup_{x\in \supp(\mu)} \norm{f_x} < \infty $, and $\displaystyle \sup_{x\in \supp(\nu)} \norm{g_x} < \infty $. 

\item[(ii)] For any $\epsilon > 0$, there exists $R >0$ such that for all $a\in X$ and all $r\geq R$
 \begin{eqnarray*}
  &\displaystyle \int_{B(a,r)^c} \int_{B(a,r)} \left\lvert  \left\langle f_x , g_y \right\rangle  \right\rvert^2 d\nu(y)d\mu(x) \leq \varepsilon^2 \, \left(\mu+\nu\right) \left( B(a,r) \right),&\\
  &\displaystyle \int_{B(a,r)^c} \int_{B(a,r)} \left\lvert  \left\langle g_x , f_y \right\rangle  \right\rvert^2 d\mu(y)d\nu(x) \leq \varepsilon^2 \, \left(\mu+\nu\right) \left( B(a,r) \right).&
  \end{eqnarray*}
 \end{itemize}

Then the continuous frames satisfy the localization property (L) from Theorem~\ref{main}.

\end{prop}

\begin{proof} To establish the localization property we will bound the expressions in the difference separately. Due to symmetry it is enough to concentrate on one of the terms.

Let $B$ be any ball in $X$.
By  Cauchy-Schwarz inequality we have
\[ \left\lvert  \int_{B^c}\int_B \left\langle g_x,f_y \right\rangle \left\langle \widetilde{f_y}, \widetilde{g_x} \right\rangle d\nu(x)d\mu(y)  \right\rvert \leq   \int_B \left( \int_{B^c} \left\lvert \left\langle g_x,f_y \right\rangle \right\rvert^2 d\mu(y) \right)^\frac{1}{2} \left( \int_{B^c} \left\lvert \left\langle \widetilde{f_y}, \widetilde{g_x} \right\rangle \right\rvert^2 d\mu(y) \right)^\frac{1}{2}  d\nu(x). \]
Applying Cauchy-Schwartz inequality again and using the fact that $\left\{ \widetilde{f_x}  \right\}_{x\in \left( X,d,\mu\right)}$ is a continuous frame for $\FF$ we obtain
\[ \left\lvert  \int_{B^c}\int_B \left\langle g_x,f_y \right\rangle \left\langle \widetilde{f_y}, \widetilde{g_x} \right\rangle d\nu(x)d\mu(y)  \right\rvert \lesssim  \left( \int_{B} \int_{B^c} \left\lvert \left\langle g_x,f_y \right\rangle \right\rvert^2  d\mu(y) d\nu(x) \int_B\norm{P_\FF \widetilde{g_x}}^2d\nu(x)\right)^\frac{1}{2}. \]
By condition (i) the continuous frame $\{ g_x  \}_{x\in \left( X,d,\nu\right)}$ is bounded and hence its canonical dual $\left\{ \widetilde{g_x}  \right\}_{x\in \left( X,d,\nu\right)}$ is also bounded. Therefore,
\[ \left\lvert  \int_{B^c}\int_B \left\langle g_x,f_y \right\rangle \left\langle \widetilde{f_y}, \widetilde{g_x} \right\rangle d\nu(x)d\mu(y)  \right\rvert \lesssim \nu(B)^\frac{1}{2}\left( \int_{B^c} \int_{B} \left\lvert \left\langle g_x,f_y \right\rangle \right\rvert^2   d\nu(x) d\mu(y) \right)^\frac{1}{2}. \]
Let $\epsilon>0$. Combining the previous inequality with (ii) we obtain $R>0$ such that for all $a\in X$ and $r>R$ we have
\[ \left\lvert  \int_{B(a,r)^c}\int_{B(a,r)} \left\langle g_x,f_y \right\rangle \left\langle \widetilde{f_y}, \widetilde{g_x} \right\rangle d\nu(x)d\mu(y)  \right\rvert \lesssim \epsilon(\mu(B(a,r))+\nu(B(a,r))).\] 
\end{proof}

We can further simplify the conditions in the previous proposition if we assume extra regularity of the indexing metric measure spaces. The extra assumption is the so called \emph{annular decay property}. 

\begin{definition} We will say that a Borel measure $\mu$ on a metric space $X$ satisfies the annular decay property if for any $a\in X$ and $\rho>0$, we have $\mu(B(a,r+\rho)\setminus B(a,r))=o(\mu(B(a,r)))$ as $r\to\infty$. 
\end{definition}

It is well-known that this condition is satisfied whenever the corresponding metric measure space is a doubling length space~\cite{Buc}. We want to note that the annular decay terminology that we use here is not  standard. We decided not to go into technicalities and took as a definition the simplest condition which is used in all of our results. 

\begin{prop}\label{Gro} Let $\FF$ and $\GG$ be closed subspaces of a Hilbert space $\HH$. Let $\left\{ f_x  \right\}_{x\in \left( X,d,\mu\right)}$ and  $\left\{ g_x  \right\}_{x\in \left( X,d,\nu\right)}$ be continuous frames for $\mathcal{F}$ and $\mathcal{G}$ respectively. Assume that the following conditions hold. 

\begin{itemize}
\item[(i)] Both continuous frames are bounded, i.e., $\displaystyle \sup_{x\in \supp(\mu)} \norm{f_x} < \infty $, and $\displaystyle \sup_{x\in \supp(\nu)} \norm{g_x} < \infty $. 

\item[(ii)]  
 \begin{eqnarray*}
  &\displaystyle \lim_{r\to\infty}\sup_{x\in X}\int_{B(x,r)^c} \left\lvert  \left\langle f_x , g_y \right\rangle  \right\rvert^2 d\nu(y)=0,&\\
  &\displaystyle \lim_{r\to\infty}\sup_{x\in X}\int_{B(x,r)^c} \left\lvert  \left\langle g_x , f_y \right\rangle  \right\rvert^2 d\mu(y)=0.&
  \end{eqnarray*}
  
 \item[(iii)] Both $\mu$ and $\nu$ satisfy the annular decay property.
  
 \end{itemize}

Then the continuous frames satisfy the localization property (L) from Theorem~\ref{main}.

\end{prop}

\begin{proof} It is enough to show that the condition (ii) in Proposition~\ref{nd} holds. Due to symmetry we can concentrate on just one of the inequalities in this condition. 

Let $\epsilon>0$. By (ii) there exists $\rho>0$ such that for all $x\in X$ \[\int_{B(x,\rho)^c} \left\lvert  \left\langle g_x , f_y \right\rangle  \right\rvert^2 d\mu(y)<\epsilon.\] By (iii) there exists $R>0$ such that for all $a\in X$ and all $r>R$ we have $\mu(B(a, r+\rho)\setminus B(a,r))<\epsilon \mu(B(a,r))$.

If $x\in B(a,r)$ then $B(a,r+\rho)^c\subset B(x,\rho)^c$. Therefore, 
\[\int_{B(a,r)}\int_{B(a,r+\rho)^c} \abs{\left\langle g_x,f_y \right\rangle}^2 d\mu(y)d\nu(x)\leq \int_{B(a,r)}\int_{B(x,\rho)^c} \abs{\left\langle g_x,f_y \right\rangle}^2 d\mu(y)d\nu(x)<\epsilon\nu(B(a,r)).\]  
On the other hand, using (i) and the fact that $\{g_x\}_{x\in (X,d,\nu)}$ is a continuous frame we obtain
\[\int_{B(a,r)}\int_{B(a,r+\rho)\setminus B(a,r)} \abs{\left\langle g_x,f_y \right\rangle}^2 d\mu(y)d\nu(x)\lesssim \int_{B(a,r+\rho)\setminus B(a,r)}\norm{f_y}^2 d\mu(y)\lesssim \epsilon\mu(B(a,r)).\]  
Combining the last two inequalities we obtain the desired inequality.

\end{proof}

\begin{rem} The theorem continues to hold if we replace condition (iii) with the following one:

(iii') $\mu$ satisfies the annular decay property and for any $a\in X$ and $\rho>0$, we have $\nu(B(a,r+\rho)\setminus B(a,r))=o(\mu(B(a,r)))$ as $r\to\infty$.  

The proof is essentially the same. We will need to use this slightly modified condition in one of the results below.
\end{rem}

Specializing to the case of two normalized continuous Parseval frames we obtain a more precise result. Namely, if these frames satisfy the localization condition (L) above, then their indexing measures must have the same density. 

\begin{cor}\label{Parseval} Let $\left\{ f_x  \right\}_{x\in \left( X,d,\mu\right)}$ and  $\left\{ g_x  \right\}_{x\in \left( X,d,\nu\right)}$ be two continuous Parseval frames for $\HH$ which are normalized, i.e., $\norm{f_x}=1$ for all $x\in \supp(\mu)$ and $\norm{g_x}=1$ for all $x\in  \supp(\nu)$. If these two frames satisfy the localization condition (L), then $D^{\pm}_\mu(\nu)=D^{\pm}_\nu(\mu)=1$.

\end{cor} 

\begin{proof} This is an immediate consequence of Theorem~\ref{main}.

\end{proof}

\section{Application: Density theorem in RKHS.}

In this section we show how our results can be used to obtain general density theorems for sampling and interpolation sequences in a large class of RKHSs. The goal of density theorems is to provide necessary conditions for sampling and interpolation of sequences in terms of appropriate densities. 

Recall that every RKHS can be viewed as a triple $(\HH, X, K)$, consisting of a Hilbert space $\HH$, a set $X$, and a function $K:X\to \HH$ which is often called a reproducing kernel. In our paper we will restrict to a class of RKHS, satisfying few additional assumptions. We note that many classical examples of RKHS satisfy these assumptions. 

\begin{itemize}
\item[(A1)] The underlying set $X$ is a metric measure space $(X, d, \sigma)$ with a metric $d$ and a Borel measure $\sigma$. 
\item[(A2)] The Hilbert space $\HH$ is isometrically embedded into $L^2(X, \sigma)$ with the embedding $f\to \ip{f}{K_x}$. In other words $$\norm{f}^2=\int_X\abs{\ip{f}{K_x}}^2d\sigma(x)=\int_X\abs{f(x)}^2d\sigma(x),$$ for all $f\in \HH$.
\item[(A3)] The metric measure space $(X, d,\sigma)$ satisfies the annular decay property.

\end{itemize}

We now show how two very recent density results~\cite{Fuhr, Aaron} can be obtained as corollaries of our results. It should be noted that these two results use a different set of assumptions and none of them implies the other one. However, they both follow from our main theorem. 

\begin{thm}
\label{grochenig}
(Theorem 2.2, \cite{Fuhr}) Let $(\HH,X,K)$ be a RKHS satisfying (A1)-(A3). Assume that in addition the following conditions hold
\begin{enumerate}[i)]
 \item There exist constants $C_1, C_2 > 0$ such that for all $x\in X$
 $$C_1 \leq \norm{K_x}^2 \leq C_2.$$
\item (Weak localization of the kernel) For every $\epsilon > 0$ there exists $R > 0$ such that
$$ \sup_{x\in X} \int_{B(x,R)^c } \llv \ip{K_x}{K_y} \rrv^2 d\sigma(y) < \epsilon^2.$$
 \item (Homogeneous approximation property) If $\Gamma \subseteq X$ is a sequence such that $\{K_\gamma\}_{\gamma\in\Gamma}$ is a Bessel sequence for $\HH$, then for every $\epsilon > 0$ there exists $R > 0$ such that
$$ \sup_{x\in X} \sum_{\gamma \in \Gamma \cap B(x,R)^c } \llv \ip{K_x}{K_\gamma} \rrv^2  < \epsilon^2.$$
 \end{enumerate}
Then, the following results hold
\begin{enumerate}[1)]
 \item If $\{K_\gamma\}_{\gamma\in\Gamma}$ is a frame for $\HH$, then
 \begin{eqnarray*}
 D^-(\Gamma) & \geq & \liminf_{r \rightarrow \infty} \inf_{a \in X} \frac{1}{\sigma\left(B(a,r)\right)} \int_{B(a,r)} \norm{K_y}^2 d\sigma(y) \\
  D^+(\Gamma) & \geq & \limsup_{r \rightarrow \infty} \sup_{a \in X} \frac{1}{\sigma\left(B(a,r)\right)} \int_{B(a,r)} \norm{K_y}^2 d\sigma(y) \\
  \end{eqnarray*}
\item If $\{K_\gamma\}_{\gamma\in\Gamma}$ is a Riesz sequence for $\HH$, then
 \begin{eqnarray*}
D^-(\Gamma) & \leq & \liminf_{r \rightarrow \infty} \inf_{a \in X} \frac{1}{\sigma\left(B(a,r)\right)} \int_{B(a,r)} \norm{K_y}^2d\sigma(y) \\
D^+(\Gamma) & \leq & \limsup_{r \rightarrow \infty} \sup_{a \in X} \frac{1}{\sigma\left(B(a,r)\right)} \int_{B(a,r)}\norm{K_y}^2 d\sigma(y) \\
  \end{eqnarray*}
\end{enumerate}
\end{thm}

\begin{proof} 

Notice first that (A2) implies that $\{K_x\}_{x\in (X,d,\sigma)}$ is a continuous Parseval frame for $\HH$. Therefore, $\widetilde{K_x}=K_x$ for all $x\in X$. Denote by $n_\Gamma$ the counting measure of the sequence $\Gamma$, i.e., $n_\Gamma(A)=\#(A\cap\Gamma)$ for every Borel set $A\subseteq X$. 

We first prove 1). Observe that $\{K_\gamma\}_{\gamma\in \Gamma}$ being a frame implies that $\{K_\gamma\}_{\gamma\in (X,d, n_\Gamma)}$ is a continuous frame for $\HH$. Take $f_x=g_x=K_x, x\in X, \mu=\sigma, \nu=n_\Gamma, \FF=\GG=\HH$. To apply Theorem~\ref{main} we first need to show that $\{f_x\}_{x\in (X,d,\mu)}$ and $\{g_x\}_{x\in (X,d,\nu)}$ satisfy the localization property (L). For this we use Proposition~\ref{Gro}. Two conditions of this proposition follow easily:  (i) follows from i) and condition (ii) is a consequence of ii) and iii). The first part of (iii') also follows immediately from (A3). The second part of (iii') follows from i) and the fact that $\{K_\gamma\}_{\gamma\in \Gamma}$ is a Bessel sequence (for more details see Lemmas 3.3 and 3.6 in~\cite{Fuhr}). Thus, by Proposition~\ref{Gro}, the localization property is satisfied. Next, since  $\{K_\gamma\}_{\gamma\in \Gamma}$ is a frame for $\HH$ we have $\ip{g_\gamma}{\widetilde{g_\gamma}}=\ip{K_\gamma}{\widetilde{K_\gamma}}\leq 1$ for all $\gamma\in \Gamma=\supp(n_\Gamma)$. Applying Theorem~\ref{main} we obtain the desired inequalities. 

We next prove 2). Let $\KK=\overline{\text{span}}\{K_\gamma : \gamma\in \Gamma\}$. Since $\{K_\gamma\}_{\gamma\in \Gamma}$ is a Riesz sequence for $\HH$ it is also a frame for $\KK$. We now take $f_x=g_x=K_x, x\in X,  \mu=n_\Gamma, \nu=\sigma, \FF=\KK, \GG=\HH.$ Similarly as in 1) we can apply Proposition~\ref{Gro} and Theorem~\ref{main} to obtain the desired inequalities. 
\end{proof}

We next prove the main result from~\cite{Aaron}. The proof there was based on Landau's method and used concentration operators and their spectral properties. Below, as usual, we will say that a sequence $\Gamma$ is sampling (interpolating resp.) if the corresponding sequence of normalized reproducing kernels $\{k_\gamma\}_{\gamma\in \Gamma}$ is a frame (Riesz sequence resp.). 

\begin{thm}
\label{Aaron_masters}
Let $(\HH,X,K)$ be a RKHS satisfying (A1)-(A2). Let $k_x$ be the normalized reproducing kernel at $x$, i.e., $k_x=K_x/\norm{K_x}$ and let $\lambda$ be the ``normalized'' measure $d\lambda(x)=\norm{K_x}^2d\sigma(x)$. Assume, in addition, that $(X, d, \lambda)$ satisfies (A3) and the following conditions hold
\begin{enumerate}[i)]
 \item (Mean value property) For every $r>0$ there exists a constant $C_r> 0$ such that for all $a\in X$ and all $f\in \HH$
 $$\abs{\ip{f}{k_a}}^2 \leq C_r\int_{B(a,r)}\abs{\ip{f}{k_x}}^2d\lambda(x).$$
\item (Localization of the kernel) For every $\epsilon>0$ there exists $R>0$ such that for all $r\geq R$ we have 
$$ \sup_{a\in X} \int_{B(a, r)^c }\int_{B(a,r)} \llv \ip{k_x}{k_y} \rrv^2 d\lambda(x)d\lambda(y) < \epsilon\lambda(B(a,r)).$$
\end{enumerate}

Let $\Gamma$ be a separated (uniformly discrete) sequence, i.e., there exists $\delta>0$ such that $d(\gamma',\gamma'')>\delta$ for all $\gamma', \gamma''\in \Gamma$. Then
\begin{enumerate}
\item If $\Gamma$ is sampling then $D^-(\Gamma)\geq 1$.
\item If $\Gamma$ is interpolating then $D^+(\Gamma)\leq 1$.
\end{enumerate}

\end{thm}

\begin{rem} Note that condition ii) says that the trace norm and the Hilbert-Schmidt norm of the so called concentration operator $T_{B(a,r)}f=\int_{B(a,r)}\ip{f}{k_x}k_xd\lambda(x)$ are asymptotically close (see~\cite{Aaron}).   

\end{rem}

\begin{proof} Notice first that (A2) implies that $\{k_x\}_{x\in (X,d,\lambda)}$ is a continuous Parseval frame for $\HH$ and $\norm{k_x}=1$ for all $x\in X$. Therefore, $\widetilde{k_x}=k_x$ for all $x\in X$. Again let $n_\Gamma$ be the counting measure of the sequence $\Gamma$, i.e., $n_\Gamma(A)=\#(A\cap\Gamma)$ for every Borel set $A\subseteq X$.     

We first prove (1). Take $f_x=g_x=k_x, x\in X, \mu=\lambda, \nu=n_\Gamma, \FF=\GG=\HH$. Then clearly $\norm{f_x}=\norm{g_x}=1$ for all $x\in X$. To apply Theorem~\ref{main} we first need to show that $\{f_x\}_{x\in (X,d,\mu)}$ and $\{g_x\}_{x\in (X,d,\nu)}$ satisfy the localization property (L). For this we use Proposition~\ref{nd}. Condition (i) in this proposition follows from the fact the continuous frames are normalized. To prove condition (ii) we need to do some work. We concentrate on proving the first inequality the other one being similar. Let $\epsilon>0$. Using the localization of the kernel ii) we get $R>0$ such that for all $r>R$ and all $a\in X$ we have 
$$\int_{B(a, r)^c }\int_{B(a,r)} \llv \ip{k_x}{k_y} \rrv^2 d\lambda(x)d\lambda(y) < \epsilon \lambda(B(a,r)).$$ Using that $\Gamma$ is separated and the mean value property i) we have

\begin{eqnarray*}
\int_{B\left( a,r \right)^c} \int_{B\left( a,r \right)} \llv\llan k_x, k_y \rran\rrv^2 d\nu(y)d\mu(x) &=& \int_{B\left( a,r \right)^c} \left( \sum_{\gamma \in B\left( a,r \right) \cap \Gamma} \llv\llan k_x, k_\gamma \rran\rrv^2 \right) d\lambda(x) \\
&\lesssim& \int_{B\left( a,r \right)^c} \left( \sum_{\gamma \in B\left( a,r \right) \cap \Gamma} \int_{B\left( \gamma, \frac{\delta}{2} \right)} \llv\llan k_x, k_z \rran\rrv^2 d\lambda(z) \right) d\lambda(x)\\
&\leq& \int_{B\left( a,r \right)^c} \int_{B\left( a, r+ \frac{\delta}{2} \right)} \llv\llan k_x, k_z \rran\rrv^2 d\lambda(z)d\lambda(x)
\end{eqnarray*}  
We split the last double integral into two parts $ \int_{B( a, r )^c} \int_{B( a, r)}+\int_{B(a,r)^c} \int_{B(a, r+ \frac{\delta}{2})\setminus B(a,r)}$. The first part is obviously bounded by $\epsilon \lambda(B(a,r))$. The second part can be bounded by \[\int_{X}\int_{B(a, r+ \frac{\delta}{2})\setminus B(a,r)} \llv\llan k_x, k_z \rran\rrv^2 d\lambda(z)d\lambda(x)=\int_{B(a, r+ \frac{\delta}{2})\setminus B(a,r)} d\lambda(z)<\epsilon\lambda(B(a,r)).\] Combining them we get the desired estimate. The other estimate can be obtained in a similar way. Therefore, by Proposition~\ref{nd}, the localization property is satisfied. Next, since  $\{k_\gamma\}_{\gamma\in \Gamma}$ is a frame for $\HH$ we have $\ip{g_\gamma}{\widetilde{g_\gamma}}=\ip{k_\gamma}{\widetilde{k_\gamma}}\leq1$ for all $\gamma\in \Gamma=\supp(n_\Gamma)$. Finally, $\ip{f_x}{\widetilde{f_x}}=\ip{k_x}{k_x}= 1$ for all $x\in X$.  Applying part (iii) of Theorem~\ref{main} we obtain $D^-(\Gamma)\geq 1$. 

We next prove (2). Let $\KK=\overline{\text{span}}\{k_\gamma : \gamma\in \Gamma\}$. Since $\{k_\gamma\}_{\gamma\in \Gamma}$ is a Riesz sequence for $\HH$ it is also a frame for $\KK$. We now take $f_x=g_x=k_x, x\in X,  \mu=n_\Gamma, \nu=\lambda, \FF=\KK, \GG=\HH.$ Similarly as in (1) we can apply Proposition~\ref{nd} and Theorem~\ref{main} to obtain $D^+(\Gamma)\leq 1.$

\end{proof}

\section{Application: Density result for embeddings} The purpose of this short section is to provide an example that shows our main theorem can be used to prove results not immediately related to sampling and interpolation. We prove (under certain technical assumptions) that if a given reproducing kernel is isometrically embedded into two different $L^2$-spaces, then the corresponding ``normalized''  measures must have equal densities. More precisely, let $(\HH,X,K)$ be a RKHS satisfying (A1)-(A2). Let $k_x$ be the normalized reproducing kernel at $x$, i.e., $k_x=K_x/\norm{K_x}$ and let $\mu$ be the corresponding ``normalized'' measure $d\mu(x)=\norm{K_x}^2d\sigma(x)$ such that the metric space $(X, d, \mu)$ satisfies (A3). Assume also that there exists another Borel measure $\alpha$ on $X$ such that $$\norm{f}^2=\int_X\abs{\ip{f}{K_x}}^2d\alpha(x)=\int_X\abs{f(x)}^2d\alpha(x),$$ for all $f\in \HH$. Denote by $\nu$ the corresponding ``normalized'' measure $d\nu(x)=\norm{K_x}^2d\alpha(x)$ and assume that the 
corresponding metric measure space $(X, d, \nu)$ satisfies (A3). Assume also that 
\begin{itemize}
\item[(i)] For every $\epsilon > 0$ there exists $R > 0$ such that
$$ \sup_{x\in X}\frac{1}{\norm{K_x}^2} \int_{B(x,R)^c } \llv \ip{K_x}{K_y} \rrv^2 d\sigma(y) < \epsilon^2.$$
 \item[(ii)] For every $\epsilon > 0$ there exists $R > 0$ such that
$$ \sup_{x\in X} \frac{1}{\norm{K_x}^2}\int_{B(x,R)^c } \llv \ip{K_x}{K_y} \rrv^2d\alpha(x) < \epsilon^2.$$
 \end{itemize}

\begin{thm}  $D^{\pm}_\mu(\nu)=D^{\pm}_\nu(\mu)=1$. 
\end{thm}

\begin{proof} This is a direct consequence of the Corollary~\ref{Parseval} applied to the continuous Parseval frames $\{k_x\}_{x\in (X, d, \mu)}$ and $\{k_x\}_{x\in (X, d, \nu)}$. The fact that these are localized follows from (i), (ii) and Proposition~\ref{Gro}. 

\end{proof}

\section{Other Applications} Several applications of theorem~\ref{grochenig} are given in~\cite{Fuhr}. As we showed above, this theorem is a consequence of our more general result and therefore all of these applications follow from our result as well. In this section we list a few additional applications of our main result. Most of these were originally proved using either Landau's spectral approach or Ramanthan-Steger comparison principle. The goal of this section is to show that our result/method can also be used to derive these results. We would like to stress that verifying the localization condition often represents a significant technical difficulty and seems to be very much case dependent.   

\subsection{Density theorem for Gabor frames} The density theorem for Gabor frames is one of the fundamental results of
time-frequency analysis with a very rich history (see~\cite{Hei} for a comprehensive treatment of the history of this problem). The most general irregular case was settled in several steps in~\cite{Lan, RS, CDH}. In fact, one of the first and main successes of the Ramanthan-Steger method~\cite{RS} was that it showed that the density theorem for Gabor frames doesn't require any (decay) conditions on the generating function.  We show here that our method can be also used to prove this result. We briefly outline the proof since many of the details are well-known. 

Let $h\in L^2(\R^n)$ and let $\rho:\R^2\to L^2(\R^n)$ be the usual projective representation of the Heisenberg group given by $\rho(p,q)f(x)=e^{2\pi i q x}f(x-p)$. Assume that $\Lambda\subset \R^{2n}$ is an uniformly discrete (separated) sequence such that the set $\{\rho(p, q)h : (p,q)\in \Lambda\}$ forms a frame for $L^2(\R^n)$. The famous density theorem for Gabor frames says that in this case the lower Beurling density of $\Lambda$ cannot be greater than $1$. We give a new proof of this result using our Theorem~\ref{main} and Proposition~\ref{Gro}. Let $\phi_0(x)=2^{n/4}e^{-\pi\abs{x}^2}$ be the standard Gaussian function.  We will use the following two continuous frames for $L^2(\R^n)$. The first being $f_{(p,q)}=\rho(p,q)h$ with the indexing metric measure space $\R^{2n}$ equipped with the Euclidean metric and the counting measure $n_\Lambda$. The second being $g_{(p,q)}=\rho(p,q)\phi_0$ with the indexing metric measure space $\R^{2n}$ equipped with the Euclidean metric and the Lebesque measure. As 
above, to apply Theorem~\ref{main} we need to establish the localization property (L) for which we use the Proposition~\ref{Gro}. Condition (i) is clearly true. The second condition follows essentially from the proof of Lemma 1 in~\cite{RS}. Finally, condition (iii') follows from the fact that $\Lambda$ is uniformly discrete. Therefore, we can apply Theorem~\ref{main} and obtain the desired lower density estimate $D^-(\Lambda)\geq 1$.

\subsection{Sampling and Interpolation in de Branges spaces.} It was proved in~\cite{MNO} that appropriate Beurling density conditions can be used to give necessary and sufficient conditions for sampling and interpolation in a large class of de~Branges spaces. Their proof of the necessity was based on the Ramanathan-Steger comparison principle. We show here how our method can be used to obtain the same conclusion. Recall that each de~Branges space is generated by a Hermitte-Biehler class function $E$ (entire function satisfying $\overline{E(\bar{z})}>E(z)$ for $z\in \C_+$) and consists of all entire functions $F:\C\to\C$ such that $$\int_\R\abs{\frac{F(x)}{E(x)}}^2dx<\infty,$$ the last expression defining the norm in the space ($\norm{f}^2$ to be precise). The phase function $\phi$ is defined from the polar representation of $E$ on the real line $E(x)=\abs{E(x)}e^{-\phi(x)}$ and is taken to be increasing. This phase function can be used to define a metric measure space by taking $X=\R$, $d(x,y)=\abs{\phi(x)-\
phi(y)}$, and $d\lambda(x)=\phi'(x)dx$. The main assumption used in~\cite{MNO} is that the measure $\lambda$ (just defined) is doubling  (which implies the annular decay property on $\lambda$). It is not hard to check that any such de~Branges space satisfies all the conditions of Theorem~\ref{Aaron_masters} (all of the technical points are essentially contained in~\cite{MNO}). Applying this theorem we obtain exactly the density conditions from~\cite{MNO}.    

\subsection{Sampling and Interpolation in weighted Fock spaces.} It was proved in~\cite{OrtSei} that in the weighted case, just as in the classical (unweighted) case, sampling and interpolation sequences in the Fock space can be completely characterized by appropriate Beurling densities. We show how to use our method to obtain the necessary part of their result. Recall that the weighted Fock space is defined as the space of entire functions $f:\C\to \C$ that satisfy the following integrability condition 
$$\int_\C\abs{f(z)}^2e^{-2\phi(z)}dm(z)<\infty,$$ where $m$ is the Lebesgue measure on $\C$ and $\phi$ is a subharmonic function satisfying the condition $\Delta\phi\simeq 1$ on $\C$. Due to the last condition on $\phi$ we can renormalize our space and use the norm $$\norm{f}^2=\int_\C\abs{f(z)}^2e^{-2\phi(z)}\Delta\phi(z),$$ without changing the sampling and interpolation sequences in the space. With this renormalization, we get a RKHS on which we can apply Theorem~\ref{Aaron_masters}. Notice that for this space the normalized measure $d\lambda$ is just $\Delta\phi$. Due to Lemma 1 and Theorem B in~\cite{OrtSei} we can apply Theorem~\ref{Aaron_masters} to obtain the desired density results. Note that our result does not give the strict inequality condition which is more closely dependent on the nature of the weighted Fock space.

\subsection{Sampling and Interpolation in the space of bandlimited functions associated to the Hankel transform.} Just as in the case of the Paley-Wiener space (with discontinuous spectrum) one can define a space of bandlimited functions associated to the Hankel transform. Here the Hankel transform plays the role that the Fourier transform plays in the Paley-Wiener space. More precisely, for a given measurable subset $S$ of the positive real axis $(0,\infty)$ consider the space $\BB_\alpha(S)$ of all functions $f\in L^2(0,\infty)$ whose Hankel transform 
$$\int_0^\infty f(t) (xt)^{1/2}J_\alpha(xt)dt,$$ is supported in $S$, where $\alpha>1/2$ and $J_\alpha$ is the classical Bessel function of order $\alpha$. It was proved in~\cite{AB} that just as in the Paley-Wiener space appropriate Beurling densities can be used to obtain necessary conditions for sampling and interpolation in $\BB_\alpha(S)$. Using the estimates from~\cite{AB} (Lemma 3) it is not hard to see that all the conditions of Theorem~\ref{Aaron_masters} apply and hence our Theorem~\ref{Aaron_masters} can be used to obtain the density results from~\cite{AB}.

\section{Final remark} Our main result (Theorem~\ref{main}) can be slightly extended in the following way. Consider a collection of vectors $\{f_x\}_{x\in (X, d, \mu)}$ in a Hilbert space $\HH$ which is not necessarily a continuous frame for any subspace of $\HH$. Let $\FF$ be the closed span of $\{f_x\}_{x\in X}$. Assume that we can find a collection $\{\tilde{f_x}\}_{x\in X}$ such that $$ P_{\mathcal{F}}f = \int_X \llan f,\tld{f_x} \rran f_x d\mu(x) = \int_X \llan f,f_x \rran \tld{f_x} d\mu(x). $$ In the case when $\{f_x\}_{x\in (X, d, \mu)}$ is a continuous frame for $\FF$ the canonical dual continuous frame has this property. However, even if $\{f_x\}_{x\in (X, d, \mu)}$ is a minimal system (so not necessarily a continuous frame) we can still take $\{\tilde{f_x}\}_{x\in X}$ to be its biorthogonal system and the projection formula will still hold. Our main result (with essentially the same proof) continues to hold under this slightly weaker initial assumption. This observation can be used to obtain 
necessary density bound for uniformly minimal systems in quite general RKHSs.

\end{document}